\documentclass[letterpaper, 10 pt, conference]{ieeeconf}  

\IEEEoverridecommandlockouts                              
\usepackage{amsmath}		
\usepackage{amscd}
\usepackage{amsfonts}
\usepackage{amssymb}		
\usepackage{mathtools}		

\usepackage[english]{babel}		
\usepackage[utf8]{inputenc}	
\usepackage{calc}
\usepackage{thumbpdf}			
\usepackage{latexsym}
\usepackage{xcolor}
\usepackage{enumerate}		
\usepackage{url}				
\usepackage{cite}	
\usepackage{graphics} 
\usepackage{epsfig} 
\usepackage{color}

\usepackage[bookmarks,bookmarksnumbered,colorlinks]{hyperref}
\hypersetup{linkcolor = black,anchorcolor = black,citecolor =
	black,filecolor = black,urlcolor = black}

\definecolor{darkgreen}{rgb}{0,0.6,0}

\newtheorem{defn}{Definition}
\newtheorem{lem}[defn]{Lemma}
\newtheorem{thm}[defn]{Theorem}
\newtheorem{cor}[defn]{Corollary}

\newtheorem{ass}[defn]{Assumption}
\newtheorem{Bem}[defn]{Remark}

\usepackage{algorithm}
\usepackage[noend]{algpseudocode}

\def\eps{\varepsilon}

\newcommand{\R}{\mathbb{R}}
\newcommand{\N}{\mathbb{N}}

\newcommand{\K}{\mathcal{K}}
\newcommand{\U}{\mathcal{U}}
\newcommand{\X}{\mathcal{X}}
\newcommand{\XX}{\mathbb{X}}

\newcommand{\W}{\mathcal{W}}
\newcommand{\B}[1]{\mathcal{B}(#1)}
\newcommand{\RR}[1]{\mathcal{R}(#1)}

\newcommand{\F}{\mathcal{F}}

\newcommand{\Prob}{\mathbb{P}}

\newcommand{\Exp}[1]{\mathbb{E}\left[#1\right]}

\title{\LARGE \bf
Near-optimal performance of stochastic economic MPC
}

\author{Jonas Schießl, Ruchuan Ou, Timm Faulwasser, Michael H. Baumann, and Lars Grüne
\thanks{
The authors gratefully acknowledge that this work was funded by the Deutsche Forschungsgemeinschaft (DFG, German Research Foundation) – project number 499435839.}
\thanks{Jonas Schießl, Michael H. Baumann and Lars Grüne are with Mathematical Institute, University of Bayreuth, Germany,
        \tt\small \{jonas.schiessl,michael.baumann,lars.gruene\} @uni-bayreuth.de}%
\thanks{Ruchuan Ou and Timm Faulwasser are the with 
Institute for Energy Systems, Energy Efficiency and Energy Economics, TU Dortmund University, Germany and with the Institute of Control Systems, Hamburg University of Technology, Hamburg, Germany,
        \tt\small ruchuan.ou@tuhh.de, timm.faulwasser@ieee.org}%
}

\begin{document}

\maketitle
\thispagestyle{empty}
\pagestyle{empty}

\begin{abstract}
This paper presents first results for near optimality in expectation of the closed-loop solutions for stochastic economic MPC. The approach relies on a recently developed turnpike property for stochastic optimal control problems at an optimal stationary process, combined with techniques for analyzing time-varying economic MPC schemes. We obtain near optimality in finite time as well as overtaking and average near optimality on infinite time horizons.
\end{abstract}

\section{INTRODUCTION}

Stochastic Model Predictive Control (MPC) has been a subject of intense research over the past years. While topics like stability, constraint satisfaction and recursive feasibility have been thoroughly investigated (see, e.g., \cite{Kouvaritakis16a,Chatterjee15a,Lucia20a} for stability or \cite{Koehler23a} for recursive feasibility), only little is known about performance of the MPC closed loop. In fact, the only results we are aware of are non-tight bounds on the expected average performance as, e.g., in \cite{Chatterjee15a}.

Motivated by the development of a stochastic dissipativity and turnpike theory in \cite{Schiessl2023a,Schiessl2023b,Schiessl2023c}, in this paper we propose an approach for analyzing the expected closed-loop performance of a stochastic economic MPC scheme. Besides certain regularity assumptions, our main structural assumption on the stochastic optimal control problem that is solved iteratively in the MPC loop is that the optimal trajectories exhibit a turnpike property. As shown in \cite{Schiessl2023a}, in stochastic optimal control problems the turnpike property means that the optimal solutions most of the time stay close to an optimal stationary process. As this optimal process is stationary, its distribution is constant over time. Nevertheless, when seen as a stochastic process, i.e., as a time-dependent function in the space of random variables, the optimal process is time varying. Therefore, one can interpret the stochastic turnpike property as a deterministic time-varying turnpike property in the space of random variables, which enables us to apply techniques for analyzing the performance of deterministic time-varying economic MPC from \cite{Gruene2017,Pirkelmann2020}. Proceeding this way, we can show that the MPC closed-loop trajectories
\begin{itemize}
    \item are initial pieces of near-optimal infinite-horizon trajectories for a suitably shifted cost function;
    \item are near overtaking optimal and near averaged optimal for the original cost function,
\end{itemize}
all in expectation. 
While these results apply to an abstract MPC algorithm that requires the knowledge of the distribution of the closed-loop solutions (which is usually not available in practice), a result that we consider interesting on its own shows that this abstract algorithm has the same expected closed-loop cost as an MPC algorithm that only uses state information sampled along a single closed-loop path and is thus practically implementable. 

The paper is organized as follows. In Section \ref{sec:2}, we define the problem setting and present background information from stochastic dissipativity and turnpike theory and from stochastic dynamic programming. In Section \ref{sec:3}, we introduce the two MPC algorithms and analyze their relation. Section \ref{sec:4} contains the performance estimates and their proofs and in Section \ref{sec:5}, we illustrate our results by numerical simulations. Section \ref{sec:6} concludes the paper.

\section{SETTING AND PRELIMINARIES}\label{sec:2}

\subsection{Problem formulation}
    For Borel spaces $\X$, $\U$, and $\W$ and a continuous function
    \begin{equation*}
        f : \X\times \U \times \W \rightarrow \X, \quad (x,u,w) \mapsto f(x,u,w)
    \end{equation*}
    we consider the discrete-time stochastic system
    \begin{equation} 
    \label{eq:stochSys}
        X(k+1) = f(X(k),U(k),W(k)), \quad X(0) = X_0.
    \end{equation}
    Here, the initial condition $X_0 \in \RR{\Omega,\X}$, the states $X(k) \in \RR{\Omega,\X}$, the controls $U(k) \in \RR{\Omega,\U}$ and the noise $W(k) \in \RR{\Omega,\W}$ are random variables on the probability space $(\Omega, \mathcal{B}(\Omega), \Prob)$ for all $k \in \N_0$, where 
    \[\RR{\Omega,\X} := \{X: (\Omega, \B{\Omega}) \rightarrow (\X,\B{\X}) \mbox{ measurable} \}\]
    and $\B{\X}$ is the Borel $\sigma$-algebra on $\X$.
    Furthermore, $W(k)$ is independent of $X(k)$ and $U(k)$ for all $k \in \N_0$ the sequence $\{W(k)\}_{k \in \N_0}$ is \emph{i.i.d.} with distribution $P_W$.
    
    Additionally, we assume that the control process $\mathbf{U} := (U(0),U(1),\ldots)$ is measurable with respect to the natural filtration $(\F_k)_{k \in \N_0}$, i.e.
    \begin{equation} \label{eq:Filtration}
        \sigma(U(k)) \subseteq \F_k := \sigma((X(0),\ldots,X(k)), \quad k \in \N_0,
    \end{equation}
    for all $k \in \N_0$. This condition can be seen as a causality requirement, formalizing that we do not use information about the future noise and only the information contained in $X(0),\ldots,X(k)$ about the past noise when deciding about our control values. For more details on stochastic filtrations we refer to \cite{Fristedt1997, Protter2005}. We call a control sequence that satisfies \eqref{eq:Filtration} admissible and denote the set of all admissible control sequences for the initial value $X_0$ on horizon $N \in \N \cup \{\infty\}$ by $\U^N(X_0)$.
    For a given initial value $X_0$ and control sequence $\mathbf{U}$, we denote the solution of system \eqref{eq:stochSys} by $X_{\mathbf{U}}(\cdot,X_0)$, or short by $X(\cdot)$ if the initial value and the control are unambiguous. 
    Note, that the solution $X_{\mathbf{U}}(\cdot,X_0)$ also depends on the disturbance $\mathbf{W}$. 
    However, for the sake of readability, we do not highlight this in our notation and assume in the following that $\mathbf{W} := (W(0),W(1),\ldots)$ is an arbitrary but fixed stochastic process.
    
    Given a continuous function $g: \X \times \U \rightarrow \R$ bounded from below and a time horizon $N \in \N \cup \{\infty\}$ we define the stage cost $\ell(X,U) := \Exp{g(X,U)}$ and consider the stochastic optimal control problem
    \begin{equation} 
    \label{eq:stochOCP}
        \begin{split}
            \min_{\mathbf{U} \in \U^{N}(X_0)} &J_N(X_0,\mathbf{U}) := \sum_{k=0}^{N-1} \ell(X(k),U(k)) \\
            s.t. ~ X(k+1) &= f(X(k),U(k),W(k)), ~ X(0) = X_0.
        \end{split}
    \end{equation}
    By $V_N(X_0) := \inf_{\mathbf{U} \in \U^{N}(X_0)} J_N(X_0,\mathbf{U})$ we denote the optimal value function of the optimal control problem \eqref{eq:stochOCP} and if a minimizer of this problem exists we will denote it by $\mathbf{U}^*_N$ or $\mathbf{U}^*_{N,X_0}$ if we want to emphasize the dependence on the initial condition. 
    Note that problem \eqref{eq:stochOCP} is well defined in the sense of \cite[Chapter~8]{Bertsekas1996b}, since it satisfies the respective assumptions. However, this notion of well definedness only ensures that the expectation $\Exp{g(X,U)}$ exists but not that $\Exp{g(X,U)} < \infty$ for all $X \in \RR{\Omega, \X}, U \in \RR{\Omega, \U}$.
    Thus, in the following, we restrict ourselves to state-control pairs $(X,U)$ with $\vert \ell(X,U) \vert < \infty$ by defining the constraint set
    \begin{equation*}
    \begin{split}
        \mathbb{X} := \{ X &\in \RR{\Omega, \X} \mid ~\exists~ U \in \RR{\Omega, \X} : \\
        &\vert \ell(X,U) \vert < \infty, ~f(X,U,W) \in \mathbb{X}\}.
    \end{split}
    \end{equation*}
    For instance, if we consider the generalized linear-quadratic problem from \cite{Schiessl2023b} with $W(k) \in L^2(\Omega;\W)$ we have  $\mathbb{X} = L^2(\Omega,\X)$ and if $\X,\U,\W$ are bounded we can choose $\mathbb{X} = \X$.  
    Note that we can also directly conclude for all $X \in \mathbb{X}$ that $\vert V_N(X) \vert < \infty$ for all finite $N \in \N$ since $V_N(X) > - \infty$ by the  boundedness assumption on $g$ and  $V_N(X) = \infty$ would be a contradiction to the existence of a control with finite cost for $X \in \mathbb{X}$.
    In order to simplify the presentation in this paper, we do not consider additional constraints. We expect that these can be added to the setup at a later stage at the expense of a more technical exposition.

\subsection{Stochastic distributional turnpike}

    For the performance estimates we aim to obtain, it is crucial that the optimal solutions of the problem \eqref{eq:stochOCP} are close to a stationary solution, except for a finite number of time instants, which is independent of the optimization horizon $N$. Here closeness is measured in an appropriate metric for stochastic systems. 
    This behavior is called the turnpike property. Before formally defining it, we have to define a proper stationary concept for stochastic systems. To this end, we make the following definition of a pair of stationary stochastic processes.
    
    \begin{defn}[Stationary pair] \label{defn:StatPair}
        A pair of the stochastic processes $(\mathbf{X}^s,\mathbf{U}^s)$ given by
        \begin{equation}
            X^s_{\mathbf{U}^s}(k+1) = f(X^s_{\mathbf{U}^s}(k), U^s(k), W(k))
        \end{equation}
        with $\mathbf{U}^s \in \U^{\infty}(X^s(0))$ is called stationary for system \eqref{eq:stochSys} if 
        \begin{equation*}
            X^s(k) \sim P^s_X, \quad U(k) \sim P^s_U,\quad  (X^s(k),U^s(k)) \sim P^s_{X,U}
        \end{equation*}
        for all $k \in \N_0$. 
    \end{defn}

    Now, we can use these stationary solutions to define a turnpike property for stochastic optimal control problems regarding the distributional behavior of the optimal solutions.
    To this end, we use the standard push-forward measure $P_X : \mathcal{B}(\X) \rightarrow [0,1]$ of the random variable $X$ with $P_X(A) := \Prob(X^{-1}(A))$ for all $A \in \mathcal{B}(\X)$
    characterizing its distribution.

    \begin{defn}[Stochastic distributional turnpike] 
    \label{defn:TurnpikeDistr}
        Consider a stationary pair $(\mathbf{X}^s,\mathbf{U}^s)$ and a metric $d_D$ on the space $\mathcal{P}(\X)$ of probability measures on $\X$. We say that the optimal control problem \eqref{eq:stochOCP} has
        \begin{enumerate}[(i)]
            \item the finite horizon distributional turnpike property if for every $X_0 \in \XX$ there exists a $\vartheta \in \mathcal{L}$ such that for each optimal trajectory $X_{\mathbf{U}^*_N}(\cdot,X_0)$ and all $N,P \in \N$ there is a set $\mathcal{Q}(X_0,L,N) \subseteq \{0,\ldots,N\}$ with $\# \mathcal{Q}(X_0,P,N) \leq P$ elements and 
            \begin{equation*}
                d_D \left( P_{X_{\mathbf{U}^*_{N}}(k,X_0)} , P^s_X \right) \leq \vartheta(P)
            \end{equation*}
            for all $k \in \{0,\ldots,N\} \setminus \mathcal{Q}(X_0,L,N)$.
            \item the infinite horizon distributional turnpike property if for every $X_0 \in \XX$ there exists a $\vartheta_{\infty} \in \mathcal{L}$ such that for each optimal trajectory $X_{\mathbf{U}^*_{\infty}}(\cdot,X_0)$ and all $N,P \in \N$ there is a set $\mathcal{Q}(X_0,L,\infty) \subseteq \{0,\ldots,N\}$ with $\# \mathcal{Q}(X_0,L,\infty) \leq L$ elements and 
            \begin{equation*}
                d_D \left( P_{X_{\mathbf{U}^*_{\infty}}(k,X_0)} , P^s_X \right) \leq \vartheta_{\infty}(L)
            \end{equation*}
            for all $k \in \N_0 \setminus \mathcal{Q}(X_0,L,\infty)$.
        \end{enumerate}
    \end{defn}
    This turnpike property states that the distributions of the optimal trajectories are most of the time close to the distributions of a stationary stochastic process.
    This means that although the optimal solutions of problem \eqref{eq:stochOCP} are not stationary, for many time instants they are approximately stationary.
    Note that although the metric $d_D$ is arbitrary in the above definition, the exact choice will in general lead to stronger or weaker statements about the distributional behavior. It is also possible to define turnpike properties for stochastic systems, which are related to the pathwise or moment-based behavior of the solutions. However, since the value of the considered stage costs depends only on the distributions, the distributional turnpike is appropriate for our analysis. For more details on stochastic turnpikes and their connection to stochastic dissipativity notions, we refer to \cite{Schiessl2023c, Schiessl2023b}.

\subsection{Stochastic dynamic programming}

    Another important tool for our estimates is the dynamic programming principle (DPP). It states that minimizing the sum of the cost on a shorter horizon $M<N$ plus the optimal value on the remaining horizon yields the same optimal value as directly minimizing the cost on the whole horizon. The following two theorems formalize this principals on finite and infinite horizon.

    \begin{thm}[Finite horizon DPP] \label{thm:FiniteDPP}
       Consider the optimal control problem \eqref{eq:stochOCP} with $N \in \N$ and $X_0 \in \XX$.
       Let $\mathbf{U}^*_N \in \U^N(X_0)$ be an optimal control sequence on horizon $N$ and define $V_0 \equiv 0$. Then for all $N \in \N$ and all $M=1,\ldots,N$ it holds that 
       \begin{equation*}
           V_N(X_0) = \sum_{k=0}^{M-1} \ell(X_{\mathbf{U}^*_N}(k), U^*_N(k)) + V_{N-M}(X_{\mathbf{U}^*_N}(M)).
       \end{equation*}
    \end{thm}

    \begin{thm}[Infinite horizon DPP] \label{thm:InfiniteDPP}
       Consider the optimal control problem \eqref{eq:stochOCP} with $N=\infty$, $X_0 \in \XX$, and $\vert V_{\infty}(X_0) \vert < \infty$. Let $\mathbf{U}^*_{\infty} \in \U^{\infty}(X_0)$ be an optimal control sequence on the infinite horizon. Then for all $M \in \N$ it holds that
       \begin{equation*}
           V_{\infty}(X_0) = \sum_{k=0}^{M-1} \ell(X_{\mathbf{U}^*_{\infty}}(k), U^*_{\infty}(k)) + V_{\infty}(X_{\mathbf{U}^*_{\infty}}(M)).
       \end{equation*}
    \end{thm}
    
    The theorems can be proven similarly to \cite[Proposition~8.2]{Bertsekas1996b} and \cite[Proposition~9.8]{Bertsekas1996b} by utilizing the fact that optimal control sequences can be represented by Markov policies, as shown in \cite[Theorem~6.2]{Altman2021}.
    For further details on dynamic programming we also refer to \cite{Bellman1966, Bertsekas1996a}.

\section{STOCHASTIC MPC-ALGORITHMS}\label{sec:3}

    Next, we introduce the stochastic MPC algorithm under consideration. Here we distinguish between two formulations---one for practical application and one for analytical purposes---whose relation we clarify in Corollary \ref{cor:mpcalg}, below.
    Algorithm~\ref{alg:stochMPC1} formulates an MPC scheme that works directly on the random variables, where we get the next state by evaluating the system dynamics \eqref{eq:stochSys} using the calculated feedback law $\mu_N$.

    \begin{algorithm}
        \caption{Stochastic MPC Algorithm (Version 1)}\label{alg:stochMPC1}
        \begin{algorithmic}
            \For{$j=0,\ldots,K$}
                \State 1.) Set $X_0 = X(j) \in \XX$.
                \State 2.) Solve the stochastic optimal control problem \eqref{eq:stochOCP}. 
                \State 3.) Apply the MPC feedback $\mu_N(X(j)) := U^*_{N,X_0}(0)$ 
                \State ~~~ to system \eqref{eq:stochSys} and get the next state $X(j+1)$.
            \EndFor
        \end{algorithmic}
    \end{algorithm}

    While we will use Algorithm~\ref{alg:stochMPC1} for our theoretical investigations, it is difficult to use it for practical purposes. The main difficulty is that we cannot measure the random variable $X(k)$ online at every step. What we can measure, however, is the realization $X(k,\omega)$ of a random variable, which leads to Algorithm~\ref{alg:stochMPC2}.

    \begin{algorithm}
        \caption{Stochastic MPC Algorithm (Version 2)}\label{alg:stochMPC2}
        \begin{algorithmic}
            \For{$j=0,\ldots,K$}
                \State 1.) Measure the state $x_j := X(j,\omega)$ and set $X_0\equiv x_j$.
                \State 2.) Solve the stochastic optimal control problem \eqref{eq:stochOCP}. 
                \State 3.) Apply the MPC feedback $\mu_N(x_j) := U^*_{N,x_j}(0)$
                \State ~~~ to system \eqref{eq:stochSys}.
            \EndFor
        \end{algorithmic}
    \end{algorithm}

   The question we need to address is how the feedback laws computed by the two algorithms are related, and whether a performance bound for Algorithm~\ref{alg:stochMPC1} would also hold for Algorithm~\ref{alg:stochMPC2}. The answer to this question relies on the following theorem, which shows that we can construct a sequence of feedback laws that is optimal for all random variable  initial values $X$ by pointwise minimization for all $x \in \X$.

    \begin{thm}[{\cite[Proposition~8.5]{Bertsekas1996b}}]
    \label{thm:PointwiseInf}
        Assume that for $k=0,\ldots,N-1$ the infimum in 
        \begin{equation} \label{eq:DPPpath}
            \inf_{u_k \in \U} \{ g(x_k,u_k) + V_{N-k-1}(f(x_k,u_k,W)) \}
        \end{equation}
        is achieved for each $x_k \in \X$ and define $\pi_k^*(x_k):=u_k^*$ for each $x_k\in\X$, where the $u_k^*\in\U$ form a measurable selection of minimizers of \eqref{eq:DPPpath}. Then $\boldsymbol{\pi}^* := (\pi_0^*,\ldots,\pi_{N-1}^*)$ defines a sequence of universally measurable functions and satisfies
        \begin{equation*}
            \inf_{\pi_k \in \Pi} \{ \Exp{g(X,\pi_k(X) + V_{N-k-1}(f(X,\pi_k(X),W))} \}
        \end{equation*}
        for all $X \in \XX$ and $k=0,\ldots,N-1$, where $\Pi$ is the set of all universally measurable functions $\pi: \X \rightarrow \U$.
    \end{thm}
    
    Observe that $\mu_N$ from Algorithm 2 is a pointwise minimizer of \eqref{eq:DPPpath} for $k=0$, hence if the selection of the minimizer is measurable in Algorithm 1 (which is, for instance, the case when the minimizer is unique for each $x_j$), then $\pi_0^*$ can be chosen as $\mu_N$ from Algorithm 2. Hence, Theorem~\ref{thm:PointwiseInf} can be used to argue that the pointwise feedback law computed by Algorithm~\ref{alg:stochMPC2} is also optimal on the space of random variables. Yet, a measure-theoretic problem remains. The problem is that the functions $\pi^*_k$ from Theorem~\ref{thm:PointwiseInf} are only universally measurable, and thus the composition $\pi^*_K(X(k))$ does not necessarily have to be measurable in the Borel sense. This means that the control $U(k) := \pi^*_k(X(k))$ is not a random variable and therefore not admissible for the problem \eqref{eq:stochOCP}. However, the next theorem shows that for any initial condition $X$ we can find another feedback sequence, which almost surely coincides with the one in Theorem~\ref{thm:PointwiseInf} and defines admissible control values for the original problem.

    \begin{thm}
    \label{thm:FeedbackAlmostSure}
        Consider the optimal sequence $\boldsymbol{\pi}^* := (\pi_0^*,\ldots,\pi_{N-1}^*)$ of uniformly measurable functions from Theorem~\ref{thm:PointwiseInf}.
        Then for all state processes $\mathbf{X} := (X(0),\ldots,X(N))$ there exists a sequence of measurable functions $\boldsymbol{\tilde{\pi}}^* := (\tilde{\pi}_0^*,\ldots,\tilde{\pi}_{N-1}^*)$ such that $\pi_k^* = \tilde{\pi}_k^*$ holds $P_{X(k)}$-almost surely and 
        \begin{equation*}
        \begin{split}
            &V_{N-k}(X(k))\\
            &=\ell(X(k),\tilde{\pi}_k^*(X(k)) + V_{N-k-1}(f(X(k),\tilde{\pi}_k^*(X(k)),W)) \\
            &= \ell(X(k),\pi_k^*(X(k)) + V_{N-k-1}(f(X(k),\pi_k^*(X(k)),W)) 
        \end{split}
        \end{equation*}
        holds for all $k=0,\ldots,N$.
    \end{thm}
    \begin{proof}
        The existence of the sequence $\boldsymbol{\tilde{\pi}}^*$ follows from \cite[Lemma~1.2]{Crauel2002}. 
        Further, the equality of the costs for $\pi_k^*$ and $\tilde{\pi}_k^*$ follows since the functions only differ on $P_{X(k)}$-null sets and the equality to the optimal value function $V_{N-k}(X(k))$ is a consequence of the fact, that a control sequence $\mathbf{U} \in \U^{N-k}$ is admissible if and only if there is a sequence $\boldsymbol{\pi} := (\pi_0,\ldots,\pi_{N-1})$ of measurable functions $\pi_k: \X \rightarrow \U$ such that $U(k) = \pi_k(X(k))$ for all $k=0,\ldots,N-1$, see \cite[Lemma~1.14]{Kallenberg2021}.
        For more details we also refer to \cite{Bertsekas1996b}.
    \end{proof}

    Theorem~\ref{thm:FeedbackAlmostSure} implies that $\tilde \pi_0^*$ coincides almost surely with the first element of a measurable sequence of optimal feedback laws for optimal control problem \eqref{eq:stochOCP} with initial condition $X(0)$ and is thus one of the possible outcomes for $\mu_N$ from Algorithm~\ref{alg:stochMPC1}. Hence, under the assumptions that the selection of the feedback value $\mu_N$ in Algorithm~\ref{alg:stochMPC2} is measurable, the feedbacks from the two algorithms coincide $P_{X(j)}$-almost surely. This immediately yields the following corollary, in which we measure the closed-loop performance of a measurable feedback law $\mu:\RR{\Omega,\X}\to\RR{\Omega,\U}$ over a horizon $K$ by
    \[ J^{cl}_K(X_0, \mu) := \sum_{k=0}^{K-1} \ell(X_{\mu}(k,X_0),\mu(X_{\mu}(k,X_0))). \]
    \begin{cor}
    Consider $N,j \in \N$ and let $\mu_N^2$ be a measurable feedback law from Algorithm~\ref{alg:stochMPC2} for $x_j=X(j,\omega)$. Then $\mu_N^2$ coincides $P_{X(j)}$-almost surely with a feedback law $\mu_N^1$ from Algorithm~\ref{alg:stochMPC1} and the identity 
    \[  J^{cl}_K(X_0, \mu_N^1) =  J^{cl}_K(X_0, \mu_N^2) \]
    holds for all $X_0\in\RR{\Omega,\X}$.
        \label{cor:mpcalg}
    \end{cor}
    
    \begin{Bem}
    (i) Corollary \ref{cor:mpcalg} states that the estimates we will derive in the remainder of this paper for the MPC closed-loop trajectory generated by the theoretical MPC Algorithm~\ref{alg:stochMPC1} will also be valid in expectation for the MPC closed-loop trajectories generated by the practically implementable MPC Algorithm~\ref{alg:stochMPC2}. As we do not assume uniqueness of the optimal controls, Algorithm~\ref{alg:stochMPC1} may not yield a unique MPC closed-loop trajectory. However, the estimates we derive below are valid for all possible closed-loop trajectories.

    (ii) In practice, only finitely many evaluations of $\mu_N=\mu_N^2$ will be performed in Algorithm~\ref{alg:stochMPC2}. Since a finite selection of values is always measurable, the requirement of $\mu_N^2$ being measurable does not have any practical implications when running Algorithm~\ref{alg:stochMPC2}.
    \end{Bem}

\section{PERFORMANCE ESTIMATES}\label{sec:4}

    This section contains our main theoretical results. We derive different types of estimates for the closed-loop performance of Algorithm~\ref{alg:stochMPC1}. To this end, we first state a couple of useful definitions for our analysis.

    \begin{defn}[Shifted cost] \label{defn:ShiftedCost}
        Let $(\mathbf{X}^s,\mathbf{U}^s)$ be a stationary pair with constant stage cost $\ell(\mathbf{X}^s,\mathbf{U}^s) := \ell(X^s(k),U^s(k))$ for all $k \in \N_0$. We define the shifted stage cost as 
        \begin{equation*}
            \hat{\ell}(X,U) := \ell(X,U) - \ell(\mathbf{X}^s,\mathbf{U}^s).
        \end{equation*}
        Moreover, for all $N \in \N \cup \{\infty\}$ we define the shifted cost functional as $\hat{J}_N(X_0,\mathbf{U}) := \sum_{k=0}^{N-1} \hat{\ell}(X(k),U(k))$
        and the corresponding optimal value function is given by $\hat{V}_N(X_0) := \inf_{\mathbf{U} \in \U^{N}(X_0)} \hat{J}_N(X_0,\mathbf{U})$.
    \end{defn}

    \begin{defn}[Optimal operation]
    \label{defn:OptOperated}
        The stochastic optimal control problem \eqref{eq:stochOCP} is called optimally operated at $(\mathbf{X}^s,\mathbf{U}^s)$ if 
        \begin{equation*}
            \liminf_{K \rightarrow \infty} \sum_{k=0}^{K-1} \Big( \ell(X_{\mathbf{U}}(k,X_0), U(k)) - \ell(\mathbf{X}^s,\mathbf{U}^s)) \Big) \geq 0
        \end{equation*}
        holds for all $X_0$ and $\mathbf{U} \in \U^{\infty}(X_0)$.
    \end{defn}

    \begin{defn}[Continuity of the optimal value function] 
    \label{defn:Continuity}
        Consider a stationary pair $(\mathbf{X}^s,\mathbf{U}^s)$ and a metric $d_D$ on the space of probability measures $\mathcal{P}(\X)$. Then we say that $\hat{V}$ is approximately distributional continuous at $P^s_X$ 
        \begin{enumerate}[(i)]
            \item on finite horizon if there exists a function $\gamma_V : \R_0^+ \times \R_0^+ \rightarrow \R_0^+$ with $\gamma_V(N,r) \to 0$ if $N \to \infty$ and $r \to 0$, and $\gamma_V(\cdot,r)$, $\gamma_V(N,\cdot)$ monotonous for fixed $r$ and $N$ such that for each $k \in \N_0$ and there is an open ball
            \[B_{\eps}(P^s_X) := \lbrace X: \Omega \rightarrow \R^n \mid d_D \left( P_X, P^s_X \right) < \eps \rbrace\]
            such that for all $X \in B_{\eps}(P^s_X)$, 
            $N \in \N$ 
            it holds that 
            \begin{equation*}
                \vert \hat{V}_N(X) - \hat{V}_N(\mathbf{X}^s) \vert \leq \gamma_V \left(N, d_D \left( P_X, P^s_X \right)\right).
            \end{equation*}
            \label{defn:ContinuityFinite}       
            \item on infinite horizon if there exists a function $\gamma_{V_\infty} \in \K_{\infty}$ such that for all $k \in \N_0$ and $\eps > 0$ there is an open ball $B_{\eps}(P^s_X)$ around $P^s_X$ such that for all $X \in B_{\eps}(P^s_X)$ 
            it holds that
            \begin{equation*}
                \vert \hat{V}_{\infty}(X) - \hat{V}_{\infty}(\mathbf{X}^s) \vert \leq \gamma_{V_\infty}\left(d_D \left( P_X, P^s_X \right)\right).
            \end{equation*} 
            \label{defn:ContinuityInfinite}  
        \end{enumerate}
    \end{defn}
    
    \begin{ass} 
    \label{ass:Asumptions1}
        We assume that there is a stationary pair $(\mathbf{X}^s,\mathbf{U}^s)$  with $\vert \ell(\mathbf{X}^s,\mathbf{U}^s)) \vert < \infty$ such that the stochastic optimal control problem \eqref{eq:stochOCP} has the following properties:
        \begin{enumerate}[(i)]
            \item It is optimally operated at $(\mathbf{X}^s,\mathbf{U}^s)$ according to Definition~\ref{defn:OptOperated}. 
            \item It has the finite and infinite distributional turnpike property from Definition~\ref{defn:TurnpikeDistr} at $(\mathbf{X}^s,\mathbf{U}^s)$ with respect to a metric $d_D$ on $\mathcal{P}(\X)$.
            \item It has a shifted value function which is approximately distributional continuous at $\mathbf{X}^s$ on finite and infinite horizon according to Definition~\ref{defn:Continuity} with respect to the same metric $d_D$.
        \end{enumerate}
    \end{ass}

    The above conditions are standard assumptions for analyzing MPC schemes, cf.\ \cite{Gruene2017, Gruene2013}, adapted to our stochastic setting. In practice, the turnpike property can be verified via numerical simulations or by a dissipativity-based analysis. In the deterministic case, under appropriate controllability assumptions it is possible to infer the optimal operation and the continuity of the optimal value function from the turnpike property, see \cite{Gruene2013}. We conjecture that this can also be done in our stochastic setting, however, the technical details of this are beyond the scope of this paper.
    The following lemma shows an important consequence of Assumption \ref{ass:Asumptions1}.

    \begin{lem} \label{lem:finiteV}
        Under Assumption \ref{ass:Asumptions1}, for all $X \in \XX$ it holds that $\vert \hat{V}_{\infty}(X) \vert < \infty$.
    \end{lem}
    \begin{proof}
        Consider $\eps$ from the continuity property on infinite horizon from Definition~\ref{defn:Continuity}(\ref{defn:ContinuityInfinite}) and pick $L \in \N$ such that $\vartheta(L) < \eps$ with $\vartheta \in \mathcal{L}$ from Definition~\ref{defn:TurnpikeDistr}. 
        Then we can conclude by the distributional turnpike property that for all $N > L$ there is $M \leq N$ such that $X_{\mathbf{U}^*_{N}}(M,X) \in B_{\eps}(P^s_X)$.
        Hence, we obtain by the continuity property that 
        \begin{equation*}
        \begin{split}
            \vert \hat{V}_{\infty}(&X_{\mathbf{U}^*_{N}}(M,X)) - \hat{V}_{\infty}(\mathbf{X}^s) \vert \\
            &\leq \gamma_{V_\infty}\left(d_D ( P_{X_{\mathbf{U}^*_{N}}(M,X)}, P^s_X )\right) < \gamma_{V_\infty}(\eps).
        \end{split}
        \end{equation*}
        Moreover, from the optimal operation we can conclude that $\hat{V}(\mathbf{X}^s) = 0$ and, thus, $\vert \hat{V}_{\infty}(X_{\mathbf{U}^*_{N}}(M,X)) \vert < \gamma_{V_\infty}(\eps)$.
        By optimality this implies
        \begin{equation*}
        \begin{split}
            \vert \hat{V}_{\infty}(X) \vert \leq& \vert \hat{J}_M(X,\mathbf{U}^*_{N}) + \hat{V}_{\infty}(X_{\mathbf{U}^*_{\infty}}(M,X)) \vert \\
            <& \vert \hat{J}_M(X,\mathbf{U}^*_{N}) \vert + \gamma_{V_\infty}(\eps) 
        \end{split}
        \end{equation*}
        which shows the claim since $\hat{J}_M(X,\mathbf{U}^*_{N})$ is finite.
    \end{proof}

    \subsection{Non-average performance}

    In this section, we aim to obtain estimates for the non-average performance of the MPC Algorithm~\ref{alg:stochMPC1}, following similar arguments as presented in \cite{Gruene2017,Pirkelmann2020}
    , adapted to our stochastic setting. 
    Since the presented results are based on the properties from Assumption~\ref{ass:Asumptions1}, we always consider them as given in the remainder of this paper. 
    We start our investigations with two preliminary lemmas that relate the cost of optimal trajectories on a given horizon to the cost evaluated on shorter horizons.
    
    \begin{lem} \label{lem:Performance1}
        For the shifted optimal value function from Definition~\ref{defn:ShiftedCost} on infinite horizon  
        \begin{equation} \label{eq:DPP1}
            \hat{V}_{\infty}(X) = \hat{J}_M(X,\mathbf{U}^*_{\infty}) + \hat{V}_{\infty}(\mathbf{X}^s) + R_1(X,M)
        \end{equation}
        holds with $\vert R_1(X,M) \vert \leq \gamma_{V_\infty}(\vartheta_{\infty}(L))$ for all $X \in \XX$, all $N \in \N$, all sufficient large $L \in \N$ and all $M \notin \mathcal{Q}(X,L,\infty)$.
    \end{lem}
    \begin{proof}
        The infinite horizon dynamic programming principle from Theorem~\ref{thm:InfiniteDPP} yields 
        $\hat{V}_{\infty}(X) = \hat{J}_M(X,\mathbf{U}^*_{\infty}) + \hat{V}_{\infty}(X_{\mathbf{U}^*_{\infty}}(M,X))$
        for all $M \in \N$. Hence, equation \eqref{eq:DPP1} holds with 
        $R_1(X,M) = \hat{V}_{\infty}(X_{\mathbf{U}^*_{\infty}}(M,X)) - \hat{V}_{\infty}(\mathbf{X}^s).$
        If we chose $L \in \N$ sufficient large such that $\vartheta_{\infty}(L) \leq \eps$ with $\vartheta_{\infty} \in \mathcal{L}$ from Definition~\ref{defn:TurnpikeDistr} and $\eps>0$ from {Definition~\ref{defn:Continuity}(\ref{defn:ContinuityInfinite})}, we can conclude that $\vert R_1(X,M) \vert \leq \gamma_{V_{\infty}}(\vartheta_{\infty}(L))$ holds for all $M \in \N$ with $M \notin \mathcal{Q}(X,L,\infty)$ by using the continuity of $\hat{V}_{\infty}$ and the turnpike property, cf.\ \cite[Lemma~2]{Gruene2017}.
    \end{proof}

    Lemma~\ref{lem:Performance1} states that the optimal cost on infinite horizon is approximately equal to the cost of the optimal trajectory evaluated only up to some appropriately chosen time index $M$. The next lemma shows, that this statement also holds for the optimal cost on finite horizons $N \in \N$.

    \begin{lem} \label{lem:Performance2}
        For the shifted optimal value function from Definition~\ref{defn:ShiftedCost} on finite horizon
        \begin{equation*} \label{eq:DPP2}
            \hat{V}_N(X) = \hat{J}_M(X,\mathbf{U}^*_{N}) + \hat{V}_{N-M}(\mathbf{X}^s) + R_2(X,M,N)
        \end{equation*}
        holds with $\vert R_2(X,M,N) \vert \leq \gamma_{V}(N-M,\vartheta(L))$ for all $X \in \XX$, all $N \in \N$, all sufficient large $L \in \N$ and all $M \notin \mathcal{Q}(X,L,N)$. 
    \end{lem}
    \begin{proof}
        The finite horizon dynamic programming principle from Theorem~\ref{thm:FiniteDPP} yields 
        $\hat{V}_{N}(X) = \hat{J}_M(X,\mathbf{U}^*_{N}) + \hat{V}_{N-M}(X_{\mathbf{U}^*_{N}}(M,X))$
        for all $M \in \{0,\ldots,N\}$. Hence, equation \eqref{eq:DPP2} holds with 
        $R_2(X,M,N) = \hat{V}_{N-M}(X_{\mathbf{U}^*_{N}}(M,X)) - \hat{V}_{N-M}(\mathbf{X}^s).$
        If we chose $L \in \N$ sufficient large such that $\vartheta(L) \leq \eps$ with $\vartheta \in \mathcal{L}$ from Definition~\ref{defn:TurnpikeDistr} and $\eps>0$ from {Definition~\ref{defn:Continuity}(\ref{defn:ContinuityFinite})}, we can conclude that $\vert R_2(X,M,N) \vert \leq \gamma_{V}(N-M,\vartheta(L))$ holds for all $M \notin \mathcal{Q}(X,L,N)$ by using the continuity of $\hat{V}_{\infty}$ and the turnpike property, cf.\ \cite[Lemma~2]{Gruene2017}.
    \end{proof}

    The next lemma shows that we can exchange the finite and infinite horizon optimal control sequences in the cost functional at the cost of a bounded error term.

    \begin{lem} \label{lem:Performance3}
        It holds that
        \begin{equation*}
            \hat{J}_M(X,\mathbf{U}^*_{\infty}) = \hat{J}_M(X,\mathbf{U}^*_{N}) + R_3(X,M,N)
        \end{equation*}
        holds with $\vert R_3(X,M,N) \vert \leq \gamma_V(N-M, \vartheta_{\infty}(L)) + \gamma_V(N-M, \vartheta(L)) + \gamma_{V_{\infty}}(\vartheta(L)) + \gamma_{V_{\infty}}(\vartheta_{\infty}(L))$ for all $N \in \N$, all sufficient large $L \in \N$, all $X \in \XX$, and all $M \in \{0,\ldots,N\} \setminus (\mathcal{Q}(X,L,N) \cup \mathcal{Q}(X,L,\infty))$.
    \end{lem}
    \begin{proof}
        Define 
        $\tilde{R}_1(X,M,N) := \hat{V}_{\infty}(X_{\mathbf{U}^*_{N}}(M,X)) - \hat{V}_{\infty}(\mathbf{X}^s)$
        for which we get by the continuity assumption and the turnpike property
        $\vert \tilde{R}_1(X,M,N) \vert \leq \gamma_{V_\infty}(\vartheta(L))$
        for all $N \in \N$, all sufficient large $L \in \N$ and $M \in \N \setminus \mathcal{Q}(X,L,N)$ similar to the proof of Lemma~\ref{lem:Performance1}. Then, we get by the infinite horizon dynamic programming principle from Theorem~\ref{thm:InfiniteDPP} that
        \begin{equation*}
        \begin{split}
            \hat{J}_M(X,\mathbf{U}^*_{\infty}) &+ \hat{V}_{\infty}(X_{\mathbf{U}^*_{\infty}}(M,X)) \\
            &\leq \hat{J}_M(X,\mathbf{U}^*_N) + \hat{V}_{\infty}(X_{\mathbf{U}^*_N}(M,X))
        \end{split}
        \end{equation*}
        and, thus, 
        \begin{equation*}
             \hat{J}_M(X,\mathbf{U}^*_{\infty}) \leq \hat{J}_M(X,\mathbf{U}^*_N) + \tilde{R}_1(X,M,N) - R_1(X,M)
        \end{equation*}
        with $\vert R_1(X,M) \vert \leq \gamma_{V_{\infty}}(\vartheta_{\infty}(L))$ for all sufficient large $L \in \N$, and all $M \notin \mathcal{Q}(X,L,\infty)$. \\
        To receive a lower bound on $\hat{J}_M(X,\mathbf{U}^*_{\infty})$ define
        \begin{equation*}
            \tilde{R}_2(X,M,N) := \hat{V}_{N-M}(X_{\mathbf{U}^*_{\infty}}(M,X)) - \hat{V}_{N-M}(\mathbf{X}^s)
        \end{equation*}
        with $\vert \tilde{R}_2(X,M,N) \vert \leq \gamma_{V}(N-M,\vartheta_{\infty}(L))$ for $M \in \N \setminus \mathcal{Q}(X,L,\infty)$ which holds by similar arguments as in Lemma~\ref{lem:Performance2} using the continuity assumption and the turnpike property.
        Additionally, we get by the finite horizon dynamic programming principle from Theorem~\ref{thm:InfiniteDPP} that
        \begin{equation*}
        \begin{split}
            \hat{J}_M(X,\mathbf{U}^*_N) &+ \hat{V}_{N-M}(X_{\mathbf{U}^*_N}(M,X)) \\
            &\leq \hat{J}_M(X,\mathbf{U}^*_{\infty}) + \hat{V}_{N-M}(X_{\mathbf{U}^*_{\infty}}(M,X))
        \end{split}
        \end{equation*}
        holds and, thus, 
        \begin{equation*}
             \hat{J}_M(X,\mathbf{U}^*_{N}) \leq \hat{J}_M(X,\mathbf{U}^*_{\infty}) + \tilde{R}_2(X,M,N) - R_2(X,M,N)
        \end{equation*}
        with $\vert R_2(X,M,N) \vert \leq \gamma_V(N-M,\vartheta(L))$ all sufficient large $L \in \N$, and all $M \notin \mathcal{Q}(X,L,N)$. \\
        Hence, we conclude for all sufficient large $L \in N$ and all $M \in \{0,\ldots,N\} \setminus (\mathcal{Q}(X,L,N) \cup \mathcal{Q}(X,L,\infty))$ that
        \begin{equation*}
            \begin{split}
                \vert R_3(X,M,N&) \vert = \vert \hat{J}_M(X,\mathbf{U}^*_N) - \hat{J}_M(X,\mathbf{U}^*_{\infty}) \vert \\
                &\leq  \gamma_{V_{\infty}}(\vartheta(L)) + \gamma_{V_{\infty}}(\vartheta_{\infty}(L)) \\
                & + \gamma_{V}(N-M,\vartheta_{\infty}(L)) + \gamma_V(N-M,\vartheta(L)),
            \end{split}
        \end{equation*}
        cf.\ \cite[Lemma~3]{Gruene2017}, which proves the claim.
    \end{proof}

    Now we can state our first main result for the closed-loop performance of the stochastic MPC Algorithm.

    \begin{thm} \label{thm:Performance4}
        There is a $\delta \in \mathcal{L}$ such that for all $L \in \N$ and each sufficient large $N \in \N$, the closed-loop cost satisfies 
        \begin{equation*}
            \hat{J}_K^{cl}(X,\mu_N) + \hat{V}_{\infty}(X_{\mu_N}(K,X)) \leq \hat{V}_{\infty}(X) + K\delta(N).
        \end{equation*}
    \end{thm}
    \begin{proof}
        Consider $X(k)$ and abbreviate $X^+(k) := f(X(k),\mu_N(X(k)),W(k))$. Using the finite horizon dynamic programming principle, cf.\ Theorem~\ref{thm:FiniteDPP}, and $\mu_N(X(k)) := U^*_{N,X(k)}(0)$ we know that 
        \begin{equation*} 
            \hat{\ell}(X(k),\mu_N(X(k))) = \hat{V}_N(X(k)) - \hat{V}_{N-1}(X^+(k))
        \end{equation*}
        holds for all $k \in \N$. By the definition of the optimal value function and the fact that $U_{N,X(k)}^*(s + 1) = U_{N-1,X^+(k)}^*(s)$ for all $s = 0,\ldots,N-2$
        we get
        \begin{equation*}
        \begin{split}
            \hat{V}&_N(X(k)) - \hat{V}_{N-1}(X^+(k)) \\
            &= \hat{J}_N(X(k), \mathbf{U}_{N,X(k)}^*) - \hat{J}_{N-1}(X^+(k),\mathbf{U}_{N-1,X^+(k)}^*) \\
            &= \hat{J}_K(X(k), \mathbf{U}_{N,X(k)}^*) - \hat{J}_{K-1}(X^+(k),\mathbf{U}_{N-1,X^+(k)}^*) 
        \end{split}
        \end{equation*}
        for all $K \in \{1,\ldots,N\}$.
        Now let $K \in \{1,\ldots,N\}$ such that $K \notin \mathcal{Q}(X,L,N) \cup \mathcal{Q}(X,L,\infty)$ and $K-1 \notin  \mathcal{Q}(X,L,N-1) \cup \mathcal{Q}(X,L,\infty)$. In each of the four sets there are at most $L$ elements, and thus, for $N > 8L$ there is at least one such $K$ with $K \leq \frac{N}{2}$. 
        This means by setting $L = \left\lfloor \frac{N-1}{8} \right\rfloor$ and applying Lemma~\ref{thm:Performance4} and Lemma~\ref{lem:Performance1}, we obtain
        \begin{equation} \label{eq:ReformStagecost}
        \begin{split}
            \hat{\ell}(X(k),\mu_N(X(k))) =& \hat{V}_{\infty}(X(k)) - \hat{V}_{\infty}(X^+(k)) \\
            &+ R_4(X(k),K,N)
        \end{split}
        \end{equation}
        with
        \begin{equation*}
        \begin{split}
            R_4(X(k)&,K,N) = - R_1(X(k),K) + R_1(X^+(k),K-1) \\
            & - R_3(X(k),K,N) + R_3(X^+(k),K-1,N-1) 
        \end{split}
        \end{equation*}
        following the same arguments as in \cite[Theorem~1]{Gruene2017}.
        In addition, from the monotonicity of $\gamma_V$ together with Lemma~\ref{lem:Performance1} and Lemma~\ref{lem:Performance3} 
        we obtain $\vert R_4(X(k),K,N) \vert \leq \delta(N)$ with
        \begin{equation*} 
        \begin{split}
            \delta(N) := & 4 \gamma_{V_\infty}(\vartheta_{\infty}(\left\lfloor \tfrac{N-1}{8} \right\rfloor)) + 2\gamma_V(\left\lfloor \tfrac{N}{2} \right\rfloor, \vartheta_{\infty}(\left\lfloor \tfrac{N-1}{8} \right\rfloor)) \\
            &+ 2\gamma_V(\left\lfloor \tfrac{N}{2} \right\rfloor, \vartheta(\left\lfloor \tfrac{N-1}{8} \right\rfloor)) + 2\gamma_{V_{\infty}}(\vartheta(\left\lfloor \tfrac{N-1}{8} \right\rfloor)).
        \end{split}
        \end{equation*}
        Applying equation \ref{eq:ReformStagecost} for $X(k) = X_{\mu_N}(k,X)$, $k=0,\ldots,K-1$ this yields
        \begin{equation*}
        \begin{split}
            &\hat{J}^{cl}_K(X,\mu_N) = \sum_{k=0}^{K-1} \hat{\ell}(X_{\mu_N}(k,X), \mu_N(X_{\mu_N}(k,X))) \\
            \leq& \sum_{k=0}^{K-1} \left( \hat{V}_{\infty}(X_{\mu_N}(k,X)) - \hat{V}_{\infty}(X_{\mu_N}(k+1,X))+ \delta(N) \right) \\
            =& \hat{V}_{\infty}(X_{\mu_N}(0,X)) - \hat{V}_{\infty}(X_{\mu_N}(K,X))+ K\delta(N) 
        \end{split}
        \end{equation*}
        which proves the claim.
    \end{proof}

    The identity from Theorem~\ref{thm:Performance4} can be interpreted as follows: If we follow the closed-loop trajectory for $K$ steps and continue with an infinite-horizon optimal trajectory, then the performance is near-optimal and the gap to optimality is bounded by $K\delta(N)$.
    
    Note that for formulating Theorem~\ref{thm:Performance4} the use of the shifted cost is unavoidable, because otherwise the infinite-horizon optimal value function will not attain a finite value, not even in the stationary solution. An alternative way of expressing (near) infinite-horizon optimality without having to shift the cost is the concept of overtaking optimality \cite{gale1967}. The next corollary reinterprets Theorem~\ref{thm:Performance4} in the overtaking sense without using the shifted cost.
    
    \begin{cor}
        For each $X \in \XX$ and $N \in \N$ the MPC closed-loop trajectory is approximately overtaking optimal, i.e., there is an $\delta \in \K_{\infty}$ such that the inequality 
        \begin{equation*}
        \begin{split}
            &\liminf_{K \rightarrow \infty} \bigg( \sum_{k=0}^{K-1} \Big(\ell(X_{\mathbf{U}}(k,X),U(k)) \\
            & \quad - \ell(X_{\mu_N}(k,X),\mu_N(X_{\mu_N}(k,X))) \Big) + K \delta(N) \bigg) \geq 0
        \end{split}
        \end{equation*}
        holds for all $\mathbf{U} \in \U^{\infty}(X)$.
        \label{cor:overtaking}
    \end{cor}
    \begin{proof}
        From Theorem~\ref{thm:Performance4} we know that $\hat{J}^{cl}_K(X,\mu_N) \leq \hat{V}_{\infty}(X) + K \delta(N)$
        since $\hat{V}_{\infty}(X_{\mu_N}(K,X) \geq 0$ due to the optimal operation in the sense of Definition~\ref{defn:OptOperated}.
        By inserting the definitions of the closed-loop cost and the optimal value function, we obtain 
        \begin{equation*}
        \begin{split}
            &\liminf_{K \rightarrow \infty} \bigg( \sum_{k=0}^{K-1} \Big(\hat{\ell}(X_{\mathbf{U}^*_{\infty}}(k,X),U^*_{\infty}(k)) \\
            & \quad - \hat{\ell}(X_{\mu_N}(k,X),\mu_N(X_{\mu_N}(k,X))) \Big) + K \delta(N) \bigg) \geq 0
        \end{split}
        \end{equation*}
        along the computations of \cite[Corollary~4.17]{Pirkelmann2020}. Because of the optimality of $\mathbf{U}^*_{\infty}$ we can further conclude that 
        \begin{equation*}
        \begin{split}
            \liminf_{K \rightarrow \infty} \bigg( \sum_{k=0}^{K-1} &\Big(\hat{\ell}(X_{\mathbf{U}}(k,X),U(k)) \\
            & \quad - \hat{\ell}(X_{\mathbf{U}^*_{\infty}}(k,X),U^*_{\infty}(k)) \Big) \bigg) \geq 0
        \end{split}
        \end{equation*}
        holds for all $\mathbf{U} \in \U^{\infty}(X)$, which implies the assertion since $\hat{\ell}(X,U) := \ell(X,U) - \ell(\mathbf{X}^s,\mathbf{U}^s)$.
    \end{proof}

    \begin{Bem}
        One might be worried by the $K$-dependence of the error terms $K\delta(N)$ in both Theorem~\ref{thm:Performance4} and Corollary~\ref{cor:overtaking}, because this implies that the deviation from the optimal cost grows linearly in $K$. However, as the accumulated non-shifted cost itself grows linearly in $K$ (except in the unlikely case when the stationary cost $\ell(\mathbf{X}^s,\mathbf{U}^s)$ equals $0$), the relative deviation to the non-shifted optimal cost is constant in $K$. 
        This is illustrated for a numerical example in Section~\ref{sec:5}, cf. Figure~\ref{fig:cost} (left).
    \end{Bem}
    
    \subsection{Average performance}

    After receiving an estimate for the non-averaged performance in Theorem~\ref{thm:Performance4}, we will use this result to obtain an upper bound for the averaged closed-loop performance, 
    defined by 
    \begin{equation*}
        \bar{J}^{cl}_{K}(X,\mu_N) :=  \frac{1}{K} J_K^{cl}(X,\mu_N).
    \end{equation*}

    \begin{thm} \label{thm:AvgPerformance}
        For $\delta \in \mathcal{L}$ from Theorem~\ref{thm:Performance4} and  all $X \in \XX$ the averaged closed-loop cost satisfies
        \begin{equation*}
            \limsup_{K \to \infty} \bar{J}^{cl}_{K}(X,\mu_N)
            \leq \ell(\mathbf{X}^s,\mathbf{U}^s) + \delta(N).
        \end{equation*}
    \end{thm}
    \begin{proof}
        From Theorem~\ref{thm:Performance4} we know that 
        \begin{equation*}
            \hat{J}_K^{cl}(X,\mu_N) \leq \hat{V}_{\infty}(X) - \hat{V}_{\infty}(X_{\mu_N}(K,X)) + K\delta(N)
        \end{equation*}
        holds which is equivalent to 
        \begin{equation*}
            \frac{1}{K} \hat{J}_K^{cl}(X,\mu_N) \leq \frac{1}{K} (\hat{V}_{\infty}(X) - \hat{V}_{\infty}(X_{\mu_N}(K,X))) + \delta(N).
        \end{equation*}
        Moreover, since $\hat{V}_{\infty}(X_{\mu_N}(K,X)) \geq 0$ due to the optimally operation, we get
        \begin{equation*}
            \frac{\hat{J}_K^{cl}(X,\mu_N)}{K} = \bar{J}_K^{cl}(X,\mu_N) - \ell(\mathbf{X}^s,\mathbf{U}^s) \leq \frac{\hat{V}_{\infty}(X)}{K} + \delta(N)
        \end{equation*}
        which yields 
        $\limsup_{K \to \infty} \bar{J}^{cl}_{K}(X,\mu_N) \leq \ell(\mathbf{X}^s,\mathbf{U}^s) + \delta(N)$
        since $\hat{V}_{\infty}(X)$ is finite by Lemma~\ref{lem:finiteV},
        and thus, proves the claim.
    \end{proof}

\section{NUMERICAL SIMULATIONS}\label{sec:5}
    To illustrate our theoretical findings, we consider the one-dimensional nonlinear optimal control problem 
    \begin{equation} \label{eq:example}
        \begin{split}
            \min_{\mathbf{U} \in \U^{N}(X_0)} J_N(X_0,\mathbf{U}) &:= \sum_{k=0}^{N-1} \Exp{X(k)^2 + 25 U(k)^2} \\
            s.t. ~ X(k+1) &= (U(k)-X(k))^2 + W(k)
        \end{split}
    \end{equation}
    where $W(k)$ follows a two-point distribution such that $W(k) = a := 1$ with probability $p_a = 0.7$ and $W(k) = b := 0.25$ with probability $p_b = 0.3$. 
    Although it is difficult to compute the optimal stationary distribution of problem \eqref{eq:example} exactly, the simulations visualized in Figure~\ref{fig:turnpike} suggest that the turnpike property holds since all possible realizations of the optimal trajectories are close to each other in the middle of the time horizons. 
    To illustrate the closed-loop performance, we approximated the expected costs by Monte-Carlo sampling using $1000$ realizations of the MPC closed-loop trajectories with initial value $X_0 = 3$ computed by Algorithm~\ref{alg:stochMPC2}. 
    Figure~\ref{fig:cost} shows the non-averaged and averaged performances of the stochastic MPC algorithm for these simulations. 
    We can observe that for sufficiently large $K$, the cumulative costs increase approximately linearly at different rates for different horizons $N$, where the difference in the rates is caused by the different $K\delta(N)$ terms from Theorem~\ref{thm:Performance4}. The dependence on $N$ is also visible for the averaged performance, where we can see that the costs converge to a neighborhood of $\ell(\mathbf{X}^s,\mathbf{U}^s) \approx 9.83$ and get closer to this value for increasing horizon length, illustrating the results from Theorem~\ref{thm:AvgPerformance}.
    
    \begin{figure}[ht]
    \begin{center}
        \includegraphics[width=0.48\textwidth]{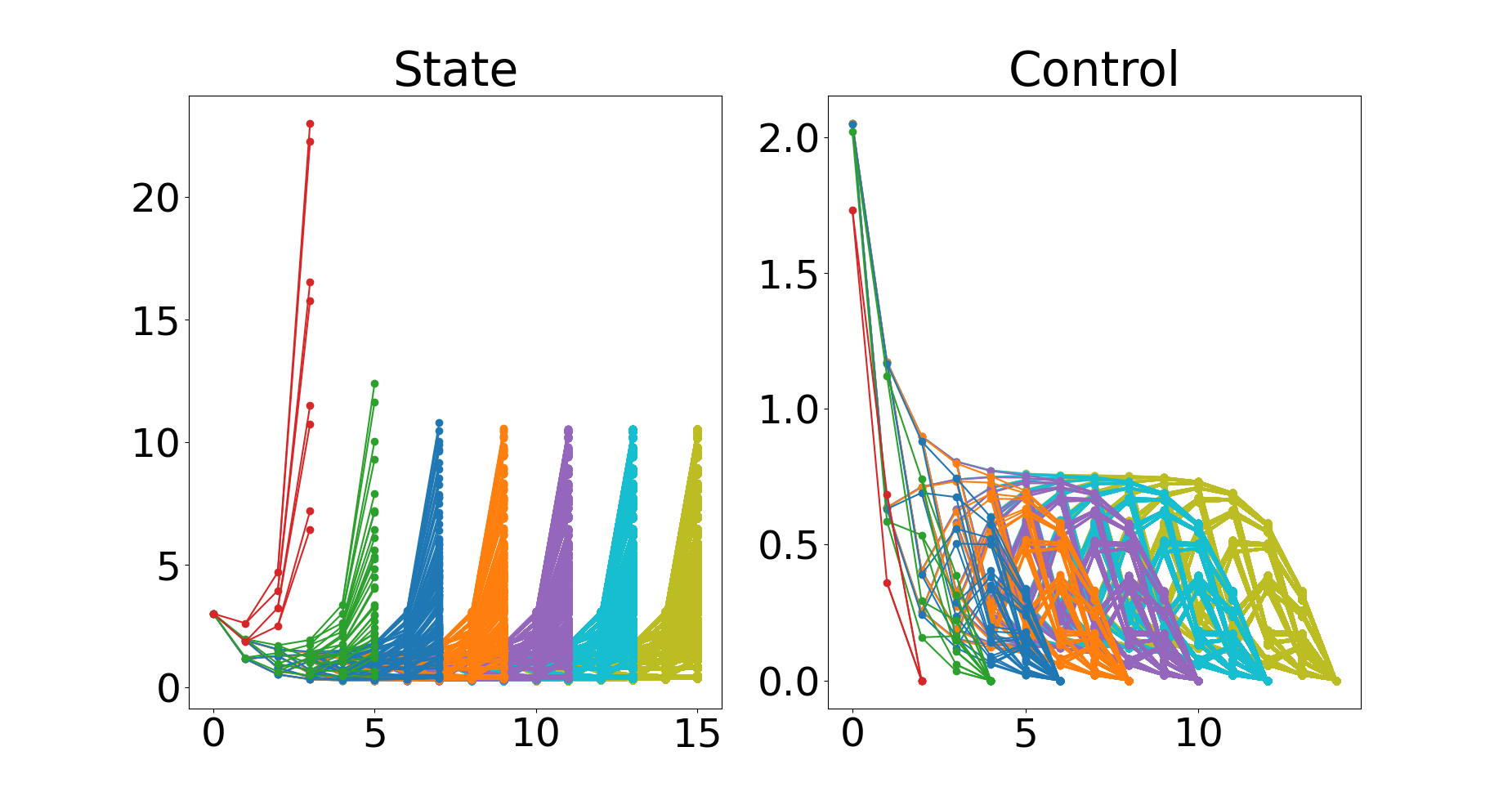}
        \vspace{-0.8cm}
        \caption{Evolution of all realization paths for the optimal trajectories (left) and controls (right) for $X_0=3$ and $N = 3,5,\ldots,15$. \label{fig:turnpike}}
    \end{center}
    \vspace{-0.2cm}
    \end{figure}
    
    \begin{figure}[ht]
    \begin{center}
        \includegraphics[width=0.48\textwidth]{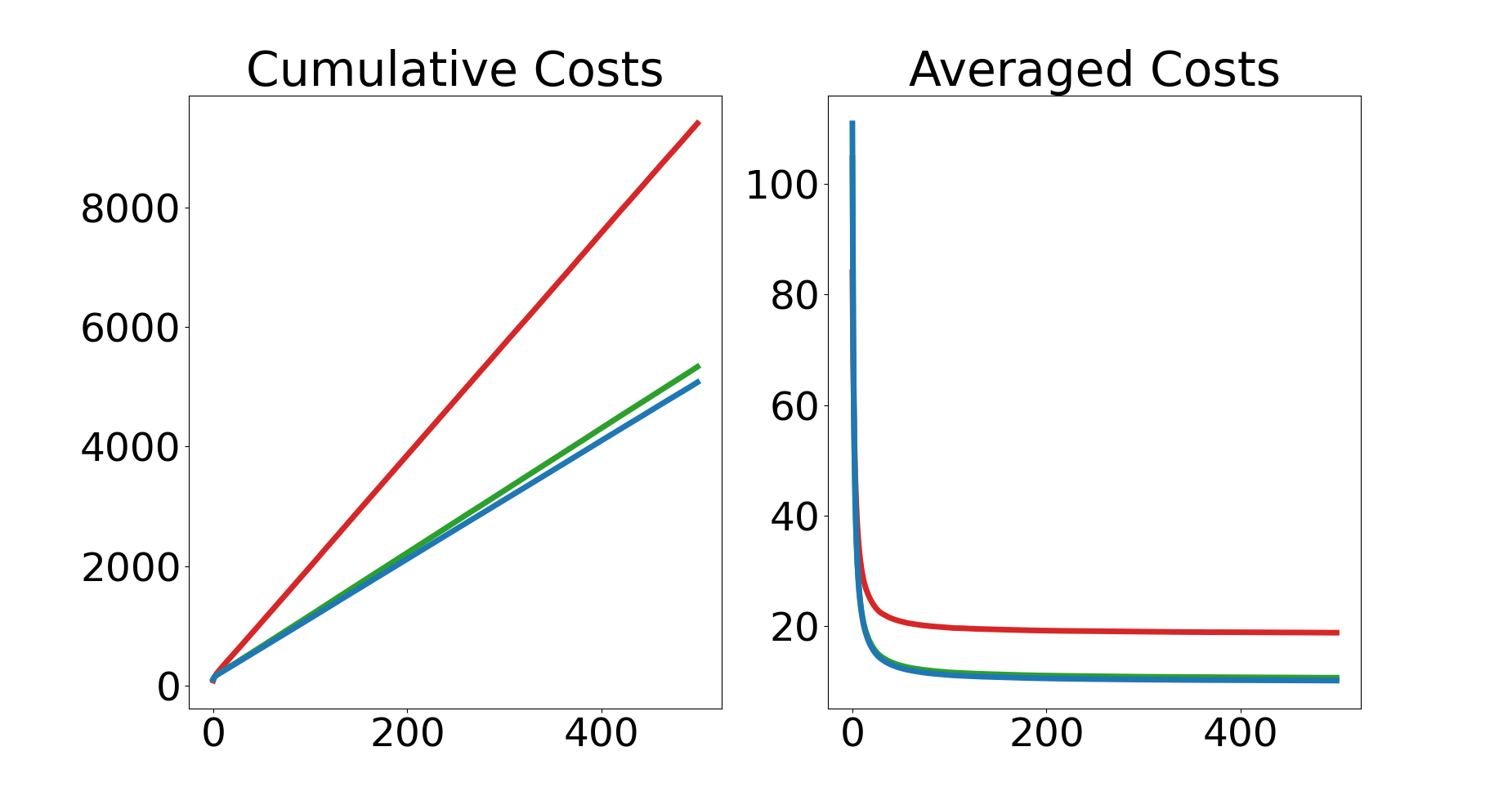}
        \vspace{-0.8cm}
        \caption{Cumulative and averaged costs of MPC Algorithm~\ref{alg:stochMPC2} for $N=3$ (red), $N=4$ (green), and $N=5$ (blue). \label{fig:cost}}
    \end{center}
    \vspace{-0.7cm}
    \end{figure}
    
\section{CONCLUSION}\label{sec:6}
We presented near-optimal performance results for a stochastic MPC scheme that is practically implementable since it uses only the state information along a sample path in each prediction step. Our investigations are based on MPC results for time-varying systems combined with turnpike concepts for stochastic optimal control problems. 
The obtained performance estimates were illustrated by a nonlinear example.

{\bibliographystyle{abbrv} 
  \bibliography{references} 
}

\end{document}